\documentclass[a4paper]{article}
\usepackage{amsmath,amssymb,amsthm,amsfonts,graphicx,fullpage,boites}
\usepackage{epsfig}
\usepackage{color}

\newtheorem{thm}{Theorem}
\newtheorem{cor}[thm]{Corollary}

\newcommand{\R}{\mathbb R}

\newcommand{\mat}[1]{\left[\begin{array}{#1}}
\newcommand{\rix}{\end{array}\right]}
\renewcommand{\t}{^{\mbox{\tiny\sf T}}}

\def\Rset{\mathbb{R}}

\def\bfo{{\mathbf 1}}

\newcommand{\comment}[1]{}
\renewcommand{\t}{^{\mbox{\tiny\sf T}}}

\usepackage{algpseudocode}
\usepackage{algorithm}
\algnewcommand{\algorithmicgoto}{\textbf{go to}}%
\algnewcommand{\Goto}[1]{\algorithmicgoto~\ref{#1}}%

\newcommand{\Span}{\mathrm{span}}
\newcommand{\rank}{\mathrm{rank}}
\newcommand{\CG}{{\tt CG}}
\newcommand{\cCG}{{\tt cCG}}
\newcommand{\mcCG}{{\tt mcCG}}

\title{
A cooperative conjugate gradient method for linear systems permitting multithread implementation of low complexity
}

\author{Amit Bhaya, Pierre-Alexandre Bliman, Guilherme Niedu and Fernando Pazos 
\thanks{ABs work was supported by grants BPP/CNPq and, additionally, CNE from FAPERJ and Universal/CNPq. GN and FP were supported by DS and PNPD fellowships, respectively, from CNPq.}
\thanks{A.\ Bhaya, G.\ Niedu, F.\ Pazos are with the Department of Electrical Engineering, Federal University of Rio de Janeiro, Rio de Janeiro, Brazil, {\tt\small amit@nacad.ufrj.br}.
P.-A.\ Bliman is with Inria, Rocquencourt BP105, 78153 Le Chesnay cedex, France, {\tt\small pierre-alexandre.bliman@inria.fr}%
}}

\begin{document}

\maketitle

\begin{abstract}
%
This paper proposes a generalization of the conjugate gradient (CG) method used to solve the equation $Ax=b$ for a symmetric positive definite matrix $A$ of large size $n$. The generalization consists of permitting the scalar control parameters (= stepsizes in gradient and conjugate gradient directions) to be replaced by matrices, so that multiple descent and conjugate directions are updated simultaneously. Implementation involves the use of multiple agents or threads and is referred to as cooperative CG (cCG), in which the cooperation between agents resides in the fact that the calculation of each entry of the control parameter matrix now involves information that comes from the other agents. For a sufficiently large dimension $n$, the use of an optimal number of cores gives the result that the multithread implementation has worst case complexity $O(n^{2+1/3})$ in exact arithmetic. Numerical experiments, that illustrate the interest of theoretical results, are carried out on a multicore computer.


\end{abstract}
\section{Introduction}

The paradigm of cooperation between agents in order to achieve some common objective has now become quite common in many areas such as control and distributed computation, while repesenting a multitude of different situations and related mathematical questions \cite{BIW2005,NO10,Murray07}.

In the field of computation, the emphasis has been mainly on the paradigm of parallel computing in which some computational task is subdivided into as many subtasks as there are available processors. The subdivision naturally induces a communication structure (or graph), connecting processors and the challenge is to achieve a subdivision that maximizes concurrency of tasks (hence minimizing total computational time), while simultaneously minimizing communication overhead. This paradigm arose as a consequence of the usual architecture of most early multiprocessor machines, in which interprocessor communication is a much slower operation than a mathematical operation carried out in the same processor. Disadvantages of this approach arise from the difficulty of effectively decomposing a large task into minimally connected subtasks, difficulties of analysis and the need for synchronization barriers at which all processors wait for the slowest one, in order to  exchange information with the correct time stamps (i.e., without asymmetric delays).

More recently, in the area of control, interest has been focused on multiagent systems, in which a number of agents cooperate amongst themselves, in a distributed manner and also subject to a communication graph that describes possible or allowable channels between agents, in order to achieve some (computational) task. Similarly, in the area of computation, multicore processors have now become common -- in these processors, each core acommodates a thread which is executed independently of the threads in the other cores. Thus, in the context of this paper, which is focused on solution of the linear system of equations $Ax=b$ for a symmetric positive definite matrix $A$ of large size $n$, we will assume that each agent carries out a task that is represented by one thread that executes on one core, so that, in this sense, the words agent and thread can be assumed to represent the same thing. In what follows, unless we are specifically talking about numerical implementation, we will give preference to the word agent.

With the advent of ever larger on-chip memory and multicore processors that allow multithread programming, it is now possible to propose a new paradigm in which each thread, with access to a common memory, computes its own estimate of the solution to the whole problem (i.e., decomposition of the problem into subproblems is avoided) and the threads exchange information amongst themselves, this being the cooperative step. The design of a cooperative algorithm has the objective of ensuring that exchanged information is used by the threads in such a way as to reduce overall convergence time.

The idea of information exchange between two iterative processes was introduced into numerical linear algebra long before the advent of multicore processors by Brezinski \cite{BR94} under the name of \textit{hybrid procedures}, defined as (we quote) ``a combination of two arbitrary approximate solutions with coefficients summing up to one...(so that) the combination only depends on one parameter whose value is chosen in order to minimize the Euclidean norm of the residual vector obtained by the hybrid procedure... The two approximate solutions which are combined in a hybrid procedure are usually obtained by two iterative methods.'' The objective of minimizing the residue is to accelerate convergence of the overall hybrid procedure. This idea was generalized and discussed in the context of distributed asynchronous computation in \cite{BBP10}.

More specifically, this paper explores the idea of cooperation between $p$ agents (or threads) in the context of the conjugate gradient (\CG) algorithm applied to an $n$-dimensional linear system $Ax=b$, for a symmetric positive definite matrix $A$ of large size $n$. Throughout the paper it is assumed that $p<n$, and even that $p\ll n$: the number of agents may be ``large", but it is usually ``much smaller" than the ``huge" size of matrix $A$. The famous \CG\ algorithm, proposed in \cite{HS52}, has several interesting properties, both as an algorithm in exact arithmetic and as one in finite precision arithmetic \cite{MS06,Gre97}. However, it is well known that, due to its structure, it cannot be parallelized in the conventional sense. In this paper, we revisit the \CG\ algorithm from a multithread perspective, which can be seen as a direct generalization of the control approach to the \CG\ algorithm proposed in \cite[pp.77-82]{BK06}, in which the scalar control parameters (stepsizes in gradient and conjugate gradient directions) are replaced by matrices (i.e., multivariable control). The cooperation between agents resides in the fact that the calculation of each entry of the control matrix now involves information that comes from the other agents. The method can also be seen as a generalization of the traditional \CG\ algorithm in which multiple descent and conjugate directions are updated simultaneously.

The paper is organized as follows.
Section \ref{se2} briefly recalls the construction, as well as the main convergence results, of Conjugate Gradient method.
Section \ref{se3} then presents the new algorithm, called {\em cooperative Conjugate Gradient (\cCG) method}.
In order to simplify this presentation of the new algorithm, the case of $p=2$ agents is first introduced in Section \ref{se331}.
The general case $p\geq 2$ is then stated in full generality in Section \ref{se332}, together with analysis results.
Complexity issues are broached in Section \ref{se4}. The results stated therein concerns execution of \cCG\ algorithm in exact arithmetic. Section \ref{se5} is then devoted to numerical experiments with the multi-thread implementation.
Section \ref{se6} provides conclusions and directions for future work.

\paragraph{Notation}
For the fixed symmetric definite positive matrix $A\in\Rset^{n\times n}$, we define $A$-norm in $\Rset^n$ by
\begin{equation}
\|x\|_A \doteq (x\t Ax)^{1/2},\quad x\in\Rset^n
\end{equation}
and define {\em $A$-orthogonality} (or {\em conjugacy}) of {\em vectors} by:
\begin{equation}
x\perp_A y \Leftrightarrow x\t Ay=0,\quad x,y\in\Rset^n\ .
\end{equation}
We will also have to consider matrices whose columns are vectors of interest.
Accordingly, we will say that $X,Y\in\Rset^{n\times p}$ are orthogonal (resp.\ $A$-orthogonal) whenever each column of $X$ is orthogonal (resp.\ $A$-orthogonal) to each column of $Y$, that is when
\begin{equation}
X\t Y =0 \qquad \text{(resp.\ $X\t AY=0$)\ .}
\end{equation}

For any set of vectors $r_i\in\Rset^n$, $i=0,1,\dots, k$, we denote respectively $\{r_i\}_0^k$ and $[r_i]_0^k$ the set of these vectors, and the matrix obtained by their concatenation: $[r_i]_0^k = \left[
\begin{matrix}
r_0 & r_1 & \dots & r_k
\end{matrix}
\right]\in\Rset^{n\times (k+1)}$.
The notation $\Span\ [r_i]_0^k$ will denote the subspace of linear combinations of the columns of the matrix $[r_i]_0^k$.
When $\Rset^n$ is the ambient vector space, we thus have
\begin{equation}
\Span\ [r_i]_0^k \doteq \left\{
v\in\Rset^n\ :\ \exists \gamma\in\Rset^{k+1},\ v = \sum_{i=0}^k \gamma_i r_i = [r_i]_0^k \gamma
\right\}\ .
\end{equation}
Similarly, for matrices $R_i\in\Rset^{n\times p}$, $i=0,\dots, k$, the notation $\{ R_i\}_0^k$ (respectively, $[R_i]_0^k\in\Rset^{n\times (k+1)p}$) is used for the set of these matrices (respectively, the matrix obtained as concatenation of the matrices $R_0, R_1,\dots, R_k$, i.e., $[R_i]_0^k = \left[
\begin{matrix}
R_0 & R_1 & \dots & R_k
\end{matrix}
\right]\in\Rset^{n\times (k+1)p}$.)
Also, we write $\Span\ [R_i]_0^k$ for the subspace of linear combinations of the columns of $[R_i]_0^k$:
\begin{equation}
\label{bracket}
\Span\ [R_i]_0^k \doteq \left\{
v\in\Rset^n\ :\ \exists \gamma\in\Rset^{(k+1)p},\ v = [R_i]_0^k \gamma
\right\}\ .
\end{equation}
Notice that this notation generalizes the definition provided earlier for vectors, and that $\Span\ [R]$ is already meaningful for a single matrix $R\in\Rset^{n\times p}$.
As an example, $\dim\Span [R] = \rank\ R$.

Last, for any matrix $R\in\Rset^{n\times p}$ and for any set $J$ of indices in $\{1, \ldots , p\}$, we will denote
\[
R|_{j\in J} \doteq (R_{ij})_{1\leq i\leq n, j\in J} . 
\]


\section{The Conjugate Gradient Method}
\label{se2}

One approach to solving the equation
\begin{equation}
\label{equa}
Ax=b\ ,
\end{equation}
with $A$
symmetric positive definite and of large dimension, is to minimize instead the convex quadratic function
\begin{equation}
\label{equb}
f(x)=\frac{1}{2}x\t Ax-b\t x\ ,
\end{equation}
since the
unique optimal point is $x^*=A^{-1}b$. Several
algorithms are based on the standard idea of generating a sequence
of points, starting from an arbitrary initial guess, and proceeding
in the descent direction (negative gradient of $f(x)$), with an
adequate choice of the step size.
In mathematical terms:
\begin{equation}
\label{eqbasica}
x_{k+1}=x_k-\alpha_kr_k,\quad r_k=\nabla f(x_k)=Ax_k-b\ ,
\end{equation}
where $\alpha_k$ is the step size.
The vector $r_k$ represents both the {\em gradient} of the cost function $f$ at the current point $x_k$, and the current {\em residue} in the process of solving \eqref{equa}.

Amongst the possible choices for $\alpha_k$, a most natural one consists in minimizing the value of the function $f$ at $x_{k+1}$, that is in taking
\begin{equation}
\alpha_k = \arg\min_{\alpha\in\Rset} f(x_k-\alpha r_k)
\end{equation}
The algorithm obtained using this principle is the {\em Steepest Descent Method}, and one shows easily that the optimal value is given by the Rayleigh quotient
\begin{equation}
\label{sdm}
\alpha_k = \frac{r_k\t r_k}{r_k\t Ar_k}
\end{equation}
Algorithm \eqref{eqbasica}-\eqref{sdm} is convergent, but in general one cannot expect better convergence speed than the one provided by
\begin{equation}
\label{speed}
\|x_k-x^*\|_A \leq \left(
\frac{\kappa-1}{\kappa+1}
\right)^k \|x_0-x^*\|_A
\end{equation}
where $\kappa$ is the condition number
\begin{equation}
\kappa \doteq \frac{\lambda_{\max} (A)}{\lambda_{\min} (A)}\ .
\end{equation}

The main weakness of Steepest Descent is the fact that steps taken in the same directions as earlier steps are likely to occur.
The {\em Conjugate Direction Methods} avoid this drawback.
Based on a set of $n$ mutually conjugate nonzero vectors $\{d_i\}_0^{n-1}$ (that is $d_i\perp_A d_j$ for any $i\neq j$, $i,j=0,\dots, n-1$), the family of conjugate direction methods use the sequence generated according to
\begin{equation}\label{itprocess}
x_{k+1}=x_k+\alpha_kx_k,\quad \alpha_k=-\frac{r_k\t d_k}{d_k\t Ad_k}\ .
\end{equation}
It is a classical result that the residue $r_{k+1}$ is orthogonal to the directions $d_i$, $i=0,\dots,k$; and that $x_{k+1}$ indeed minimizes $f$ on the affine subspace $x_0+\Span\ [d_i]_0^k$ \cite{LY08}. As a consequence of this last property, conjugate direction methods lead to finite time convergence (in exact arithmetic).

The {\em Conjugate Gradient method}, developed by Hestenes and Stiefel \cite{HS52}, is the particular method of conjugate directions obtained when constructing the conjugate directions by Gram-Schmidt orthogonalization, achieved at step $k+1$ on the set of the gradients $\{r_i\}_0^k$.
A key point here is that this construction can be carried out iteratively.
The iterative equations of the Conjugate Gradient method are given in the pseudocode instructions of Algorithm \ref{algo1}.
Instructions \ref{CGb1}--\ref{CGb2} constitute the optimal descent process in the direction $d_k$; while instructions \ref{CGc1}--\ref{CGc2} achieve iteratively the orthogonalization of the subspaces $\Span\ [r_i]_0^k$.
\begin{algorithm}                   
\caption{\tt Conjugate Gradient (CG) algorithm}
\label{algo1}
\begin{algorithmic}[1]
\tt
\State choose $x_0\in\Rset^n$
\State $r_0 := Ax_0 - b$
\State $d_0 := r_0$
\State $k := 0$
\While{$d_k\neq 0$}
	\State $\alpha_k := -r_k\t d_k (d_k\t Ad_k)^{-1}$ \label{CGb1}
	\State $x_{k+1} := x_k + \alpha_kd_k$ \label{CGb2}
	\State $r_{k+1} := Ax_{k+1} - b$
	\State $\beta_k := -r_{k+1}\t Ad_k (d_k\t Ad_k)^{-1}$ \label{CGc1}
	\State $d_{k+1} := r_{k+1} + \beta_kd_k$ \label{CGc2}
	\State $k \gets k+1$
\EndWhile
\end{algorithmic}
\end{algorithm}

We recall the main properties of this algorithm, in an adapted form, to allow for easier comparison with the results to be stated later.
\begin{thm}[Properties of CG]
\label{thCG}
As long as the vector $d_k$ is not zero
\begin{itemize}
\item
the vectors $\{r_i\}_0^k$ are mutually orthogonal, the vectors $\{d_i\}_0^k$ are mutually $A$-orthogonal, and the subspaces $\Span\ [r_i]_0^k, \Span\ [d_i]_0^k$ and $\Span\ [A^ir_0]_0^k$ are equal and have dimension $(k+1)$;
\item
the point $x_{k+1}$ is the minimizer of $f$ on the affine subspace $x_0+\Span\ [d_i]_0^k$.
\end{itemize}
\end{thm}
When the residue vector is zero, the optimum has been attained, showing that CG terminates in finite time.
Apart from the finite time convergence property, the following formula indicates net improvement with respect to Steepest Descent:
\begin{equation}
\label{speed2}
\|x_k-x^*\|_A \leq 2 \left(
\frac{\sqrt{\kappa}-1}{\sqrt{\kappa}+1}
\right)^k \|x_0-x^*\|_A
\end{equation}
which represents substantial improvement with respect to \eqref{speed}.

For a proof of this theorem as well as further details on the contents of this section, see \cite{LY08,Gul10}.


\section{Statement and Analysis of the Cooperative Conjugate Gradient Method}
\label{se3}


\subsection{The two-agent case}
\label{se331}

In this subsection, in order to aid comprehension and ease notation, the case of two agents (the case $p=2$) is considered: their estimates at step $k$ are written as $x_k, x'_k$ respectively, the residues as $r_k, r'_k$, and the two descent directions as $d_k, d'_k$.
The gradients at each one of the current estimates are given as
\begin{equation}
\begin{pmatrix}
r_k & r'_k
\end{pmatrix}
=
A\begin{pmatrix}
x_k & x'_k
\end{pmatrix}
-\bfo\t b\ ,
\end{equation}
with $\bfo =\begin{pmatrix} 1 \\ 1 \end{pmatrix}$.
As for \CG\ method, we distinguish two steps.

\paragraph{$\bullet$ Descent step.}

Given the current residues $r_k, r'_k$ and two descent directions $d_k, d'_k$, this step determines the upgraded value of the estimates $x_{k+1}, x'_{k+1}$ and therefore of the residues $r_{k+1}, r'_{k+1}$.

One allows the use of the two descent directions $d_k, d'_k$, thus looking for updates of the form
\begin{equation}
\begin{pmatrix}
x_{k+1} & x'_{k+1}
\end{pmatrix}
= \begin{pmatrix}
x_k & x'_k
\end{pmatrix}
+ \begin{pmatrix}
d_k & d'_k
\end{pmatrix} \alpha_k\t \ .
\end{equation}
The matrix $\alpha_k\in\Rset^{2\times 2}$ has to be chosen.
In the same spirit as for \CG, this choice is made in such a way as to minimize $f(x_{k+1})$ and $f(x'_{k+1})$.
This yields in fact two independent minimization problems.
Denoting
\begin{equation}
\alpha_j \doteq \begin{pmatrix} \alpha_{j1} & \alpha_{j2} \end{pmatrix},\quad j=1,2\ ,
\end{equation}
the two optimality conditions are given by
\begin{equation}
0 =
\begin{pmatrix}
d_k & d'_k
\end{pmatrix}\t A \left(
x_k + \begin{pmatrix}
d_k & d'_k
\end{pmatrix}\alpha_1\t
\right) - \begin{pmatrix}
d_k & d'_k
\end{pmatrix}\t b =
\begin{pmatrix}
d_k & d'_k
\end{pmatrix}\t A \left(
x'_k + \begin{pmatrix}
d_k & d'_k
\end{pmatrix}\alpha_2\t
\right) - \begin{pmatrix}
d_k & d'_k
\end{pmatrix}\t b\ .
\end{equation}
This shows that the minimum is uniquely defined, and attained when
\begin{equation}
\label{eq3}
\alpha_k
= \begin{pmatrix}
\alpha_1 \\ \alpha_2
\end{pmatrix}
= - \begin{pmatrix}
r_k & r_k'
\end{pmatrix}\t
\begin{pmatrix}
d_k & d_k'
\end{pmatrix}
(\begin{pmatrix}
d_k & d_k'
\end{pmatrix}\t A \begin{pmatrix}
d_k & d_k'
\end{pmatrix})^{-1}\ .
\end{equation}
Notice that the two descent directions $d_k, d'_k$ have to be linearly independent for the matrix in \eqref{eq3} to be invertible.
Similarly to \CG\ algorithm, we have the following \textit{four} useful properties
\begin{equation}
\label{eq35}
r_{k+1}, r'_{k+1} \perp d_k, d'_k\ .
\end{equation}

\paragraph{$\bullet$ Orthogonalization step.}

The second step consists, given the residues $r_{k+1}, r'_{k+1}$ and the current descent directions $d_k, d'_k$, in determining the next descent directions $d_{k+1}, d'_{k+1}$.
The latter should be $A$-orthogonal to all the previous descent directions.
In fact, it will be sufficient to ensure $A$-orthogonality to $d_k, d'_k$, as for \CG.
One takes
\begin{equation}
\label{eq5}
\begin{pmatrix}
d_{k+1} & d'_{k+1}
\end{pmatrix}
= \begin{pmatrix}
r_{k+1} & r'_{k+1}
\end{pmatrix}
+ \begin{pmatrix}
d_k & d'_k
\end{pmatrix} \beta_k\t \ .
\end{equation}
The matrix $\beta_k\in\Rset^{2\times 2}$ is chosen to ensure the {\em four} conditions
\[
d_{k+1}, d'_{k+1} \perp_A d_k, d'_k\ .
\]
This also leads to two independent problems for the two vectors $d_{k+1}, d'_{k+1}$: writing now
\begin{equation}
\beta_j \doteq \begin{pmatrix} \beta_{j1} & \beta_{j2} \end{pmatrix},\quad j=1,2\ ,
\end{equation}
the previous orthogonality conditions can be written as:
\begin{equation}
0 = \left(
r_k +  \begin{pmatrix}
d_k & d'_k
\end{pmatrix} \beta_1\t
\right)\t A  \begin{pmatrix}
d_k & d'_k
\end{pmatrix}
= \left(
r'_k +  \begin{pmatrix}
d_k & d'_k
\end{pmatrix} \beta_2\t
\right)\t A  \begin{pmatrix}
d_k & d'_k
\end{pmatrix}
\end{equation}
which yields the unique solution
\begin{equation}
\label{eq7}
\beta_k
= \begin{pmatrix}
\beta_1 \\ \beta_2
\end{pmatrix}
= - \begin{pmatrix}
r_k & r'_k
\end{pmatrix}\t A \begin{pmatrix}
d_k & d'_k
\end{pmatrix} (\begin{pmatrix}
d_k & d'_k
\end{pmatrix}\t A \begin{pmatrix}
d_k & d'_k
\end{pmatrix})^{-1}
\end{equation}

\paragraph{$\bullet$ Summary of \cCG\ in the case $p=2$.}

Putting together the previous findings, we summarize the \cCG\ algorithm in the case $p=2$ as
\begin{subequations}
\label{cCG2}
\begin{gather}
\label{cCG2a}
\begin{pmatrix}
r_k & r'_k
\end{pmatrix} = A\begin{pmatrix}
x_k & x'_k
\end{pmatrix} -\bfo\t b\\
\label{cCG2b}
\begin{pmatrix}
x_{k+1} & x'_{k+1}
\end{pmatrix} = \begin{pmatrix}
x_k & x'_k
\end{pmatrix}+\begin{pmatrix}
r_k & r'_k
\end{pmatrix}\alpha_k\t\\
\label{cCG2b2}
\alpha_ k
= - \begin{pmatrix}
r_k & r_k'
\end{pmatrix}\t
\begin{pmatrix}
d_k & d_k'
\end{pmatrix}
(\begin{pmatrix}
d_k & d_k'
\end{pmatrix}\t A \begin{pmatrix}
d_k & d_k'
\end{pmatrix})^{-1} \\
\label{cCG2c}
\begin{pmatrix}
d_{k+1} & d_{k+1}'
\end{pmatrix} = \begin{pmatrix}
r_{k+1} & r_{k+1}'
\end{pmatrix}+ \begin{pmatrix}
d_k & d_k'
\end{pmatrix}\beta_k\t \\
\label{cCG2c2}
\beta_k =-\begin{pmatrix}
r_{k+1} & r_{k+1}'
\end{pmatrix}\t A \begin{pmatrix}
d_k & d_k'
\end{pmatrix} (\begin{pmatrix}
d_k & d_k'
\end{pmatrix}\t A \begin{pmatrix}
d_k & d_k'
\end{pmatrix})^{-1}
\end{gather}
\end{subequations}

\subsection{Cooperative CG algorithm: the general case}
\label{se332}

We now provide generalization to the case of $p\geq 2$ agents.
The extension is indeed straightforward from \eqref{cCG2}.
The matrices whose $j$-th column represents respectively the solution estimate, the residue and the descent direction of agent $j$, $j=1,\dots, p$, for iteration $k$ are denoted $X_k\in\Rset^{n\times p}$, $R_k\in\Rset^{n\times p}$, $D_k\in\Rset^{n\times p}$.
In other words, $X_k, R_k, D_k$ stand for the matrices written down $\begin{pmatrix}
x_k & x_k'
\end{pmatrix}, \begin{pmatrix}
r_k & r_k'
\end{pmatrix}, \begin{pmatrix}
d_k & d_k'
\end{pmatrix}$ in Section \ref{se331}.

The algorithm \cCG\ in full generality is given as the list of instructions in Algorithm \ref{algo2}.
Algorithm \cCG\ is a generalization of \CG\ , which is the case $p=1$. In all algorithms in this paper, comments appear to the right of the symbol $\triangleright$.
\begin{algorithm}                    
\caption{\tt cooperative Conjugate Gradient (cCG) algorithm}
\label{algo2}
\begin{algorithmic}[1]
\tt
\State choose $X_0\in\Rset^{n\times p}$
\State $R_0 := AX_0 - \bfo_p\t b$
\State $D_0 := R_0$
\color{blue}
\Comment{Generically, $\rank\ D_0 = p$}
\color{black}
\State $k := 0$
\While{$D_k$ is full rank}
	\State $\alpha_k := -R_k\t D_k (D_k\t AD_k)^{-1}$ \label{cCGb1}
\color{blue}
	\Comment{$\alpha_k\in\Rset^{p\times p}$}
\color{black}
	\State $X_{k+1} := X_k + D_k\alpha_k\t$ \label{cCGb2}
\color{blue}
	\Comment{$X_{k+1}\in\Rset^{p\times p}$}
\color{black}
	\State $R_{k+1} := AX_{k+1} - \bfo_{p_k}\t b$ \label{cCGa}
\color{blue}
	\Comment{$R_{k+1}\in\Rset^{p\times p}$}
\color{black}
	\State $\beta_k := -R_{k+1}\t AD_k (D_k\t AD_k)^{-1}$ \label{cCGc1}
\color{blue}
	\Comment{$\beta_k\in\Rset^{p\times p}$}
\color{black}
	\State $D_{k+1} := R_{k+1} + D_k\beta_k\t$ \label{cCGc2}
\color{blue}
	\Comment{$D_{k+1}\in\Rset^{p\times p}$}
\color{black}
	\State $k \gets k+1$
\EndWhile
\end{algorithmic}
\end{algorithm}


\begin{thm}[Properties of \cCG]
\label{thcCG}
As long as the matrix $D_k$ is full rank (that is $\rank\ D_k =p$)
\begin{itemize}
\item
the matrices $\{R_i\}_0^k$ are mutually orthogonal, the matrices $\{D_i\}_0^k$ are mutually $A$-orthogonal, and the subspaces $\Span\ [R_i]_0^k, \Span\ [D_i]_0^k$ and $\Span\ [A^iR_0]_0^k$ are equal and have dimension $(k+1)p$;
\item
for any vector $e_j$ of the canonical basis of $\Rset^p$, the vector $X_{k+1}e_j\in\Rset^n$ (which constitutes the $j$-th column of $X_{k+1}$) is the minimizer of $f$ on the affine subspace $X_0e_j+ \Span\ [D_i]_0^k$.
\end{itemize}
\end{thm}
Theorem \ref{thcCG} indicates that, as long as the residue vector $R_k$ is full rank, the algorithm \cCG\ behaves essentially as does \CG, providing $p$ different estimates at iteration $k$, each of them being optimal in an affine subspace constructed from one of the $p$ initial conditions and the common vector space obtained from the columns of the direction matrices $D_i$, $i=0,\dots, k-1$. This vector space, $\Span\ [D_i]_0^k$, has dimension $(k+1)p$: each iteration involves the cancellation of $p$ directions. Notice that different columns of the matrices $D_k$ are not necessarily $A$-orthogonal (in other words, $D_k\t AD_k$ is not necessarily diagonal), but, when $R_k$ is full rank, they constitute a set of $p$ independent vectors. The statement as well as the proof of this theorem are inspired by the corresponding ones for the conventional \CG\ algorithm given in \cite[p.~270ff]{LY08} and \cite[p.~390-391]{Gul10}.

\begin{proof}[Proof of Theorem \ref{thcCG}]
\mbox{}

\noindent $\bullet$
We first show that for any $k$,
\begin{equation}
\label{RD}
\Span\ [R_i]_0^k = \Span\ [D_i]_0^k\ .
\end{equation}


\noindent $\bullet$
We show the first point by induction.
Clearly, $D_j\perp_A D_k$ for any $j<k$, and $\Span\ [R_i]_0^k= \Span\ [D_i]_0^k= \Span\ [A^iR_0]_0^k$ when $k=0$, while nothing has to be verified for the orthogonality conditions.
Assume that, for some $k$, they are verified for any $i\leq k$, let us show them for $k+1$, assuming that $R_{k+1}$ is full rank.

From lines \ref{cCGb2}--\ref{cCGa} of Algorithm \ref{algo2}, $R_{k+1} = AX_{k+1}-\bfo_p\t b = R_k + AD_k\alpha_k\t$.
By induction, the columns of both the matrices $R_k$ and $AD_k$ are located in $\Span\ [A^iR_0]_0^{k+1}$.
Thus,
\[
R_{k+1} \in \Span\ [A^iR_0]_0^{k+1}
\]
and consequently
\[
\Span\ [R_i]_0^{k+1}\subset \Span\ [A^iR_0]_0^{k+1}\ .
\]

On the other hand, for any vector $e_j$ of the canonical basis of $\Rset^p$,
\begin{equation}
R_{k+1}e_j\not\in\Span\ [D_i]_0^k
\end{equation}
because each residue is orthogonal to the previous descent directions, so that $R_{k+1}e_j\in\Span\ [D_i]_0^k$ for some $e_j$ would imply $R_{k+1}e=0$, which contradicts the assumption of full rankness of $R_{k+1}$.
Indeed, for the same reason, one can also show that, for any $v\in\Rset^p\setminus\{ 0\}$, $R_{k+1}v\not\in\Span\ [D_i]_0^k$.
Using again the fact that $\Span\ [R_{k+1}] = p$, one sees that
\[
\Span\ [R_i]_0^{k+1}\supset \Span\ [A^iR_0]_0^{k+1}
\]
and indeed
\[
\Span\ [R_i]_0^{k+1}= \Span\ [A^iR_0]_0^{k+1}\ .
\]

One shows similarly from lines \eqref{cCGc1}--\eqref{cCGc2} of Algorithm \ref{algo2} that
\[
\Span\ [D_i]_0^{k+1}\subset \Span\ [A^iR_0]_0^{k+1}\ ,
\]
and the equality is obtained using the same rank argument.

The dimension of these sets is $(k+2)p$, as they contain the $p$ independent vectors in $\Span\ [D_{k+1}]$ orthogonal to $\Span\ [D_i]_0^k$.

From line \ref{cCGc2} of Algorithm \ref{algo2}, one gets
\begin{equation}
\label{A-orth}
D_i\t AD_{k+1} = D_i\t A\left(
R_{k+1} +D_k\beta_k\t
\right)
\end{equation}
For $i<k$, the first term is zero because $AD_i \in\Span\ [D_j]_0^{i+1}$ and the gradients constituting the columns of $R_{k+1}$ are orthogonal to any vector in $\Span\ [D_j]_0^{i+1}$; while the second term is also zero due to the induction hypothesis.
For $i=k$, the right-hand side of \eqref{A-orth} is zero because $\beta_k$ is precisely chosen to ensure this property. Thus the $\{D_i\}_0^{k+1}$ are mutually $A$-orthogonal.

Orthogonality of $R_{k+1}$ follows from lines \ref{cCGb1} and \ref{cCGa} of Algorithm \ref{algo2}. The induction hypothesis has been proved for $k+1$ concluding the proof of the first part of Theorem \ref{thcCG}.

\noindent $\bullet$ [Optimality]. Arguing as in \cite[p.390-391]{Gul10}, we can write
\begin{equation}
\label{opt0}
	X_{k+1} e_j = X_0 e_j + \sum_{i=0}^k D_i \gamma_i \t
\end{equation}
for some $\gamma_i \in \R^{1 \times p}$. Optimality implies that
\begin{equation}
\label{opt1}
D_i\t (AX_{k+1} e_j - b) = 0 ,
\end{equation}
Substituting (\ref{opt0}) in (\ref{opt1}) and rewriting in terms of the matrices $R_i$ and $\gamma_i$ yields
\begin{equation}
\label{opt2}
\gamma_i = -e_j \t R_0 \t D_i (D_i \t AD_i )^{-1}
\end{equation}
On the other hand,
\begin{equation}
\label{opt3}
	\nabla f (X_{k+1}e_j) = A (X_0 e_j + \sum_{i=0}^k D_i \gamma_i \t) - b = R_0 e_j + \sum_{i=0}^k AD_i \gamma_i \t
\end{equation}
Thus 
\begin{equation}
\label{opt4}
D_{k+1} \t \nabla f (X_{k+1}e_j) = D_{k+1} \t R_0 e_j
\end{equation}
Substituting (\ref{opt4}) in (\ref{opt2}), we get:
\begin{equation}
\label{opt5}
\gamma_{k+1} = -\nabla f (X_{k+1}e_j) \t (D_{k+1}\t AD_{k+1})^{-1} = e_j \t \alpha_{k+1}
\end{equation}
\end{proof}
A natural question is now to study the cases where at some point of the execution of the algorithm \cCG\, one gets $\rank\ R_k<p$.
In the best case, this occurs because one of the columns of $R_k$ is null, say the $j$-th one, meaning that $\nabla f(X_ke_j)=0$, and thus that the $k$-th estimate of the $j$-th agent is equal to the optimum $x^*=A^{-1}b$.
But, of course, $\rank\ R_k$ can be smaller than $p$ without any column of $R_k$ being null.

First of all, the following result ensures that this rank degeneracy is, in general, avoided during algorithm execution.
\begin{thm}[Genericity of the full rank condition of \cCG\ residues matrix]
\label{thGen}
For an open dense set of initial conditions $X_0$ in $\Rset^{n\times p}$, one has during any \cCG\ run
\begin{equation}
\forall\ 0\leq k\leq k^* \doteq \lfloor \frac{n}{p} \rfloor,\quad \rank\ R_k=p\ .
\end{equation}
Moreover
\begin{equation}
\label{dim}
\dim\Span\ [ D_i]_0^{k^*} = p\lfloor \frac{n}{p} \rfloor\ .
\end{equation}
Otherwise said: generically, algorithm \cCG\ can be run during $k^*$ steps, and
\begin{itemize}
\item
any of the columns of $X_{k^{*}}$ minimizes $f$ on an affine subspace of $\Rset^n$ of codimension $p \lfloor \frac{n}{p} \rfloor$;
\item
application of \CG\ departing from any of the columns of $X_{k^*}$ yields convergence in at most $n-p\lfloor \frac{n}{p} \rfloor \leq p-1$ steps.
\end{itemize}
\end{thm}

The second part of Theorem \ref{thGen} has to interpreted as follows.
When the size $n$ of the matrix $A$ is a multiple of the number $p$ of agents, then \cCG\ generically ends up in $\frac{n}{p}$ steps.
When this is not the case, the estimates $X_{k^*}$ obtained for $k=k^*$ minimize the function $f$ on affine subspace whose underlying vector subspace is $\Span\ [D_i]_0^{k^*}$ (see Theorem \ref{thcCG}).
The interest of \eqref{dim} is to show that this subspace is quite large: its codimension is $p\lfloor \frac{n}{p} \rfloor$, which is at most equal to $p-1$.

\begin{proof}[Proof of Theorem \ref{thGen}]
The main point consists in showing that generically, $k^*$ iterations of the algorithm \cCG\ can be conducted without occurrence of the rank deficiency condition. As a matter of fact, the other results of the statement are direct consequences of this fact.

To show the latter, use is made of Theorem \ref{thcCG}. From the properties stated therein, one sees that, for any $0\leq k\leq k^*$, the rank of $X_k$ is deficient if and only if a linear combination of the $p$ column vectors of $X_0$ pertains to the $kp$-dimensional subspace $\Span\ [D_i]_0^{k-1}$.
In a vector space of dimension $n> kp$, this occurs only in the complement of an open dense set.

Now, if the column vectors of $X_k$ are linearly independent, the same is true for $R_k$, see line \ref{cCGa} of Algorithm \ref{algo2}, and as well for $D_k$, see line \ref{cCGc2}. This completes the proof of Theorem \ref{thGen}.

\end{proof}

We now study what can be done in case of rank degeneracy.
When $p_k\doteq \rank\ D_k$ is such that $0< p_k <p$, this means that trajectories initially independent have come to a point where the estimates in $X_k$ will converge along directions which are now linearly dependent.
The natural solution is then to choose any full-rank subset of trajectories.
We thus propose the modified algorithm \ref{algo3}.

\begin{thm}[Convergence of \mcCG\ algorithm]
\label{thmcCG}
For any nonzero initial condition $X_0$, algorithm \mcCG\ ends up in $k^{**}$ iterations for some $k^{**}\leq n$.
Moreover
\begin{itemize}
\item
the sequence $(p_k)_{0\leq k\leq k^{**}}$ is nonincreasing;
\item
any of the columns of $X_{k^{**}}$ minimizes $f$ on an affine subspace of $\Rset^n$ of codimension $p_{k^{**}} \lfloor \frac{n}{p_{k^{**}}} \rfloor$;
\item
application of \CG\ departing from any of the columns of $X_{k^{**}}$ yields convergence in at most $n-p_{k^{**}}\lfloor \frac{n}{p_{k^{**}}} \rfloor \leq p_{k^{**}}-1$ steps.
\end{itemize}
\end{thm}

The proof is straightforward and omitted for brevity.

\color{black}
\begin{algorithm}                    
\caption{\tt modified cooperative Conjugate Gradient (mcCG) algorithm}
\label{algo3}
\begin{algorithmic}[1]
\tt
\State choose $X_0\in\Rset^{n\times p}$
\State $R_0 := AX_0 - \bfo_p\t b$
\State $D_0 := R_0$
\State $p_0 := \rank\ D_0$
\color{blue}
\Comment{$p_0$ is the initial value of the rank}
\color{black}
\State $k := 0$
\If{$p_0=p$}
\State \Goto{while}
\Else
\color{black}
		\State choose $J\in\{1,\dots, p\}$ such that $\rank\ R_0|_{j\in J} =p_0$
		\State $X_0 \gets X_0|_{j\in J}$
\color{blue}
		\Comment{$X_0\in\Rset^{p_0\times p_0}$}
\color{black}
		\State $R_0 \gets R_0|_{j\in J}$
\color{blue}
		\Comment{$R_0\in\Rset^{p_0\times p_0}$ is full rank}
\color{black}
		\State $D_0 \gets D_0|_{j\in J}$
\color{blue}
		\Comment{$D_k\in\Rset^{p_0\times p_0}$ is full rank}
\color{black}

\EndIf
\While{$p_k>0$} \label{while}
	\State $\alpha_k := -R_k\t D_k (D_k\t AD_k)^{-1}$
\color{blue}
	\Comment{$\alpha_k\in\Rset^{p_k\times p_k}$}
\color{black}
	\State $X_{k+1} := X_k + D_k\alpha_k\t$
\color{blue}
	\Comment{$X_{k+1}\in\Rset^{p_k\times p_k}$}
\color{black}
	\State $R_{k+1} := AX_{k+1} - \bfo_{p_k}\t b$
\color{blue}
	\Comment{$R_{k+1}\in\Rset^{p_k\times p_k}$}
\color{black}
	\State $\beta_k := -R_{k+1}\t AD_k (D_k\t AD_k)^{-1}$
\color{blue}
	\Comment{$\beta_k\in\Rset^{p_k\times p_k}$}
\color{black}
	\State $D_{k+1} := R_{k+1} + D_k\beta_k\t$
\color{blue}
	\Comment{$D_{k+1}\in\Rset^{p_k\times p_k}$}
\color{black}
	\State $p_{k+1} := \rank\ D_{k+1}$
	\If{$p_{k+1} = p_k$}
\color{blue}
		\Comment{If $p_{k+1} = p_k$, \cCG\ goes on normally}
\color{black}
		\State \Goto{iter}
	\Else
\color{blue}
		\Comment{If $p_{k+1} < p_k$, $p_k-p_{k+1}$ agents are suppressed}
\color{black}
		\State choose $J\in\{1,\dots, p_k\}$ such that $\rank\ R_{k+1}|_{j\in J} =p_{k+1}$
		\State $X_{k+1} \gets X_{k+1}|_{j\in J}$
\color{blue}
		\Comment{$X_{k+1}\in\Rset^{p_{k+1}\times p_{k+1}}$}
\color{black}
		\State $R_{k+1} \gets R_{k+1}|_{j\in J}$
\color{blue}
		\Comment{$R_{k+1}\in\Rset^{p_{k+1}\times p_{k+1}}$ is full rank}
\color{black}
		\State $D_k \gets D_k|_{j\in J}$
\color{blue}
		\Comment{$D_k\in\Rset^{p_{k+1}\times p_{k+1}}$ is full rank}
\color{black}
	\EndIf
	\State $k \gets k+1$ \label{iter}
\EndWhile
\end{algorithmic}
\end{algorithm}


\section{Computational complexity}
\label{se4}

This section is concerned with the evaluation of the gain in computation time of the numerical solution of equation \eqref{equa}, when using \cCG\ algorithm with $p$ agent, i.e., the gain which is expected is due to the parallelism induced by a multithread implementation.
We evaluate this issue here {\em assuming computations in exact arithmetic}.
Moreover, thanks to Theorem \ref{thGen}, we adopt the generic assumption that the rank of the residue matrices remains constant (and full), and that the computations are then carried out for $\lfloor \frac{n}{p} \rfloor$ iterations.
Disregarding as marginal the supplementary \CG\ steps (see the statement of Theorem \ref{thGen}), we thus consider it to be realistic to quantify the worst case complexity by evaluating {\em the numbers of multiplications involved by $\frac{n}{p}$ iterations of \cCG}.
Recall that the case $p=1$ corresponds to the usual \CG\ algorithm.

We propose the multithread implementation detailed in Table 1.
\begin{center}
\begin{table*}[ht]
{\scriptsize
\hfill{}
\caption{Number of scalar multiplications in $k$-th iteration}

\vspace{.4cm}
\hspace{0.4cm}
 \begin{tabular}{|c|c|c|c|}
 \hline Operation carried out by& Composite result & Dimension & Number of scalar\\
$i$-th processor  &&of the
 result& multiplications carried out by the\\
       &            &      & $i$th processor \\
 \hline \hline
 $AD_{k,i}$& $AD_k$ &$n\times p$&$n^2$\\
 \hline \hline $D_{k,i}\t AD_k$&$D_k\t AD_k$&$p\times p$&$np$\\
 \hline \hline $R_{k,i}\t D_k$&$R_k\t D_k$&$p\times p$&$np$\\
 \hline $\alpha_{k,i} \text{ s.t.\ } \alpha_{k,i}(D_k\t AD_k)=R_{k,i}\t D_k$
 &$\alpha_k =R_k\t D_k(D_k\t AD_k)^{-1}$&$p\times p$&$\frac{p(p+1)(2p+1)}{6}$\\
 \hline $R_{k+1,i} = R_{k,i}-AD_k\alpha_{k,i}\t$&$R_{k+1} =R_k -AD_k\alpha_k\t$&$n\times
 p$&$np$\\
 \hline $X_{k+1,i} = X_{k,i}- D_k\alpha_{k,i}\t$&$X_{k+1} =X_k - D_k\alpha_k\t$&$n\times
 p$&$np$\\
 \hline $R_{k+1,i}\t AD_k$&$R_{k+1}\t AD_k$&$p\times p$&$np$\\
 \hline $\beta_{k,i} \text{ s.t.\ } \beta_{k,i}(D_k\t AD_k)= -R_{k+1,i}\t AD_k$
  &$\beta_k =-R_{k+1}\t AD_k(D_k\t AD_k)^{-1}$&$p\times p$&$\frac{p(p+1)(2p+1)}{6}$\\
\hline $D_{k+1,i} = R_{k+1,i}+D_k\beta_{k,i}\t$&$D_{k+1}=R_{k+1}+D_k\beta_k\t$&$n\times p$&$np$\\
\hline\hline \multicolumn{3}{|c|}{Total number of scalar multiplications per processor and
per iteration}& $n^2+6np+\frac{p(p+1)(2p+1)}{3}$\\
\hline
\end{tabular}}\label{tab}
\hfill{}
\end{table*}
\end{center}
In Table 1, the first column indicates the task carried out at each stage by every processor, and the last column the corresponding number of multiplications carried out by a processor. The double lines, separating the first row from the second and the second from the third, indicate the necessity of a phase of information exchange: every processor at that stage needs to know results from other processors, also called a synchronization barrier in computing terminology. The second column, labelled composite result, contains the information that is available by pooling the partial results from each processor and the third column gives the dimension of this composite result. The fourth and final column contains the number of multiplications carried out by the $i$th processor. The number $\frac{p(p+1)(2p+1)}{6}$ of scalar multiplications is needed to realize Gaussian elimination realized through LU factorization \cite[p.~ 15]{Str88}.

As indicated by the last line of Table 1, a total of $n^2+6np+\frac{p(p+1)(2p+1)}{3}$ multiplications per processor is needed to complete an iteration.
Since, generically speaking, the algorithm ends in at most $\frac{n}{p}$ iterations (see Theorem \ref{thGen}), an estimate of the worst-case multithread execution time is given by the following result.
\begin{thm}[Worst-case multithread execution time in exact arithmetic]
Generically, multithread execution of \cCG\ using $p$ agents for a linear system (\ref{equa}) of size $n$ requires
\begin{equation}\label{worst_case}
N(p) = \frac{n^3}{p} +6n^2 + n\frac{(p+1)(2p+1)}{3}
\end{equation}
multiplications performed synchronously in parallel by each processor.
\end{thm}

This result has straightforward consequences.

\begin{cor}[Multithread gain]
\label{cor1}
For problems of size $n$ at least equal to 5, it is always beneficial to use $p\leq n$ processors rather than a single one.
In other words, when $n\geq 5$,
\begin{equation}
\forall\ 1 \leq p \leq n,\quad
N(1)\geq N(p)\ .
\end{equation}
\end{cor}

\begin{proof}
One has
\[
N(1)-N(p)
= n^3+\frac{5}{3}n - \left(
2n^2+\frac{2}{3}n^3
\right)
= \frac{1}{3}(n^3-6n^2+5n)
= \frac{1}{3}n (n-1)(n-5)\ .
\]
Moreover,
\[
\frac{dN(1)}{dp} = n \left(
\frac{7}{3}-n^2
\right)
\]
which is negative for $n\geq 2$, while
\[
\frac{dN(n)}{dp} = n \left(
\frac{4}{3}n^2
\right) \geq 0 \ .
\]
The convexity of $N$ then yields the conclusion that $N(p)\leq N(1)$ for any $1 \leq p \leq n$.
\end{proof}

\begin{cor}[Optimal multithread gain]
\label{cor2}
For any size $n$ of the problem, there exists a unique optimal number $p^*$ of processors minimizing $N(p)$.
Moreover, when $n\to +\infty$,
\begin{subequations}
\begin{gather}
\label{appra}
p^* \approx \left(
\frac{3}{4}
\right)^{\frac{1}{3}} n^{\frac{2}{3}}\\
\label{apprb}
N(p^*) \approx \left(
\left(
\frac{4}{3}
\right)^{\frac{1}{3}}
+ \frac{2}{3}\left(
\frac{3}{4}
\right)^{\frac{2}{3}}
\right) n^{2+\frac{1}{3}} \approx 1.651 n^{2+\frac{1}{3}}
\end{gather}
\end{subequations}
\end{cor}

\begin{proof}
One has
\[
\frac{dN(p)}{dp} = -\frac{n^3}{p^2} + \frac{4}{3} np + n\ .
\]
There exists a unique $p^*$ canceling this expression.
For this value, one has $n^2 = p^2 (\frac{4}{3}p+1)$, which yields the asymptotic behavior given in \eqref{appra}.
The value in \eqref{apprb} is directly deduced.
\end{proof}

The conclusion of Corollary \ref{cor2} is quite important.
It shows that solution of $Ax=b$ is possible by the method proposed here with a cost of $O(n^{2+\frac{1}{3}})$ multiplications.
This is to be compared with the classical results \cite{TW84}. 



\section{Numerical experiments with discussion of multithread implementation}
\label{se5}
This section reports on a suite of numerical experiments carried out on a set of random symmetric matrices of dimensions varying from $1000$ to $25000$, the latter being the largest dimension that could be accommodated in the fast access RAM memory of the multicore processor. The random symmetric matrices were generated by choosing random diagonal matrices $\Lambda$, with positive diagonal entries uniformly distributed between $1$ and a prespecified condition number, which were then pre-multiplied (resp. post-multiplied) by a random orthogonal matrix $U$ (resp. its transpose $U\t$). The random orthogonal matrices $U$ were generated using a C translation of Shilon's MATLAB code \cite{Shilon06}, which produces a matrix distribution uniform over the manifold of orthogonal  matrices with respect to the induced $\R^{n^2}$ Lebesgue measure. The right hand sides and initial conditions were also randomly generated, with all entries uniformly distributed on the interval $[-10,10]$. In this preliminary work, the matrices used were dense and the use of preconditioners was not investigated.

In order to evaluate the performance of the algorithm proposed in this paper, a program was written in language C. The compiler used was the GNU Compiler Collection (GCC), running under Linux Ubuntu 10.0.4. For the Linear Algebra calculations, we used the Linear Algebra Package (LAPACK) and the Basic Linear Algebra Subprograms (BLAS). Finally, to parallelize the program, we used the Open Multi Processing (OMP) API. The processor used was an Intel Core2Quad CPU Q8200 running at $2.33$ MHz with four cores.

The pseudo-code in Algorithm \ref{algo_imp} gives details of the implementation for three ($p=3$) agents.

\begin{algorithm}                    
\caption{\tt Implementation of cooperative Conjugate Gradient (cCG) algorithm}
\label{algo_imp}
\begin{algorithmic}[1]
\tt
\State choose $X_0, Y_0, Z_0 \in\Rset^{n}$
\color{blue}
\Comment{All initialized randomly with numbers between $-10$ and $10$}
\color{black}
\State $r_{0,x} := A\cdot x_0 - b$
\State $r_{0,y} := A\cdot y_0 - b$
\State $r_{0,z} := A\cdot z_0 - b$
\State $d_{0,x} := r_{0,x}$
\State $d_{0,y} := r_{0,y}$
\State $d_{0,z} := r_{0,z}$
\State $k := 0$
\State $minres := min(norm(r_{0,x}),norm(r_{0,y}),norm(r_{0,z}))$
\While{$minres > tolerance$}
    \color{blue}
    \State \Comment Compute matrix-vector products $A \cdot d_{k,i}$
    \color{black}
    \State {\bf agent $1$}: compute $A\cdot d_{k,x}$
    \State {\bf agent $2$}: compute $A\cdot d_{k,y}$
    \State {\bf agent $3$}: compute $A\cdot d_{k,z}$
    \State Barrier
    \color{blue}
    \Comment Synchronizes all $3$ agents, before proceeding to the next computations
    \State \Comment Compute $m_{ij}$
    \color{black}
    \State {\bf agent $1$}: $m_{11} := d_{k,x}^{T}\cdot A\cdot d_{k,x}$; $m_{12} := d_{k,x}^{T}\cdot A\cdot d_{k,y}$
    \State {\bf agent $2$}: $m_{13} := d_{k,x}^{T}\cdot A\cdot d_{k,z}$; $m_{22} := d_{k,y}^{T}\cdot A\cdot d_{k,y}$
    \State {\bf agent $3$}: $m_{23} := d_{k,y}^{T}\cdot A\cdot d_{k,z}$; $m_{33} := d_{k,z}^{T}\cdot A\cdot d_{k,z}$
    \State Barrier
    \color{blue}
    \Comment Synchronizes all $3$ agents, before proceeding to the next computations
    \color{black}
    \State Initialize $M := \{m_{ij}\}$
    \color{blue}
    \Comment Symmetric matrix needed to compute alpha, $m_{ij} = m_{ji}$
    \State \Comment Right-hand sides needed to compute alpha
    \color{black}
    \State {\bf agent $1$}: $n_{1} := [r_{k,x}^{T}\cdot d_{k,x};r_{k,x}^{T}\cdot d_{k,y};r_{k,x}^{T}\cdot d_{k,z}] $
    \State {\bf agent $2$}: $n_{2} := [r_{k,y}^{T}\cdot d_{k,x};r_{k,y}^{T}\cdot d_{k,y};r_{k,y}^{T}\cdot d_{k,z}] $
    \State {\bf agent $3$}: $n_{3} := [r_{k,z}^{T}\cdot d_{k,x};r_{k,z}^{T}\cdot d_{k,y};r_{k,z}^{T}\cdot d_{k,z}] $
    \color{blue}
    \State \Comment Computation of alpha
    \color{black}
    \State {\bf agent $1$}: Solve $M\cdot \alpha_{1} = n_{1}$
    \State {\bf agent $2$}: Solve $M\cdot \alpha_{2} = n_{2}$
    \State {\bf agent $3$}: Solve $M\cdot \alpha_{3} = n_{3}$
    \color{blue}
    \State \Comment Update estimates of each agent
    \color{black}
    \State {\bf agent $1$}: $x_{k} \leftarrow x_{k}+\alpha_{1,1}\cdot d_{k,x}+\alpha_{1,2}\cdot d_{k,y}+\alpha_{1,3}\cdot d_{k,z}$
    \State {\bf agent $2$}: $y_{k} \leftarrow y_{k}+\alpha_{2,1}\cdot d_{k,x}+\alpha_{2,2}\cdot d_{k,y}+\alpha_{2,3}\cdot d_{k,z}$
    \State {\bf agent $3$}: $z_{k} \leftarrow z_{k}+\alpha_{3,1}\cdot d_{k,x}+\alpha_{3,2}\cdot d_{k,y}+\alpha_{3,3}\cdot d_{k,z}$
    \color{blue}
    \State \Comment Update residues of each agent
    \color{black}
    \State {\bf agent $1$}: $r_{k,x} := A\cdot x_{k} - b$
    \State {\bf agent $2$}: $r_{k,y} := A\cdot y_{k} - b$
    \State {\bf agent $3$}: $r_{k,z} := A\cdot z_{k} - b$
    \color{blue}
    \State \Comment Right-hand sides needed to compute beta
    \color{black}
    \State {\bf agent $1$}: $n_{1} := [r_{k,x}^{T}\cdot A\cdot d_{k,x};r_{k,x}^{T}\cdot A\cdot d_{k,y};r_{k,x}^{T}\cdot A\cdot d_{k,z}] $
    \State {\bf agent $2$}: $n_{2} := [r_{k,y}^{T}\cdot A\cdot d_{k,x};r_{k,y}^{T}\cdot A\cdot d_{k,y};r_{k,y}^{T}\cdot A\cdot d_{k,z}] $
    \State {\bf agent $3$}: $n_{3} := [r_{k,z}^{T}\cdot A\cdot d_{k,x};r_{k,z}^{T}\cdot A\cdot d_{k,y};r_{k,z}^{T}\cdot A\cdot d_{k,z}] $
    \color{blue}
    \State \Comment Computation of beta
    \color{black}
    \State {\bf agent $1$}: Solve $M\cdot \beta_{1} = n_{1}$
    \State {\bf agent $2$}: Solve $M\cdot \beta_{2} = n_{2}$
    \State {\bf agent $3$}: Solve $M\cdot \beta_{3} = n_{3}$
    \color{blue}
    \State \Comment Update of directions
    \color{black}
    \State {\bf agent $1$}: $d_{k,x} \leftarrow r_{k,x}+\beta_{1,1}\cdot d_{k,x}+\beta_{1,2}\cdot d_{k,y}+\beta_{1,3}\cdot d_{k,z}$
    \State {\bf agent $2$}: $d_{k,y} \leftarrow r_{k,y}+\beta_{2,1}\cdot d_{k,x}+\beta_{2,2}\cdot d_{k,y}+\beta_{2,3}\cdot d_{k,z}$
    \State {\bf agent $3$}: $d_{k,z} \leftarrow r_{k,z}+\beta_{3,1}\cdot d_{k,x}+\beta_{3,2}\cdot d_{k,y}+\beta_{3,3}\cdot d_{k,z}$
    \color{blue}
    \State \Comment Calculate of residual norms
    \color{black}
    \State {\bf agent $1$}: $norm_{r_{x}} = norm(r_{k,x})$
    \State {\bf agent $2$}: $norm_{r_{y}} = norm(r_{k,y})$
    \State {\bf agent $3$}: $norm_{r_{z}} = norm(r_{k,z})$
    \State $minres := min(norm_{r_{x}},norm_{r_{y}},norm_{r_{z}})$
    \State $ k \leftarrow k+1$
\EndWhile
\end{algorithmic}
\end{algorithm}




\subsection{Evaluating speedup}
The results of the Cooperative 3 agent \cCG, in comparison with classic \CG, with a tolerance of $10^{-3}$, and matrices with different sizes, but all with the same condition number of $10^{6}$, are shown in Figure 1. Multiple tests were performed, using different randomly generated initial conditions (20 different initial conditions for the small matrices and 10 for the bigger ones). Figure 1 shows the mean values computed for these tests.

\begin{figure}[h]
\centering
\includegraphics[width=12cm]{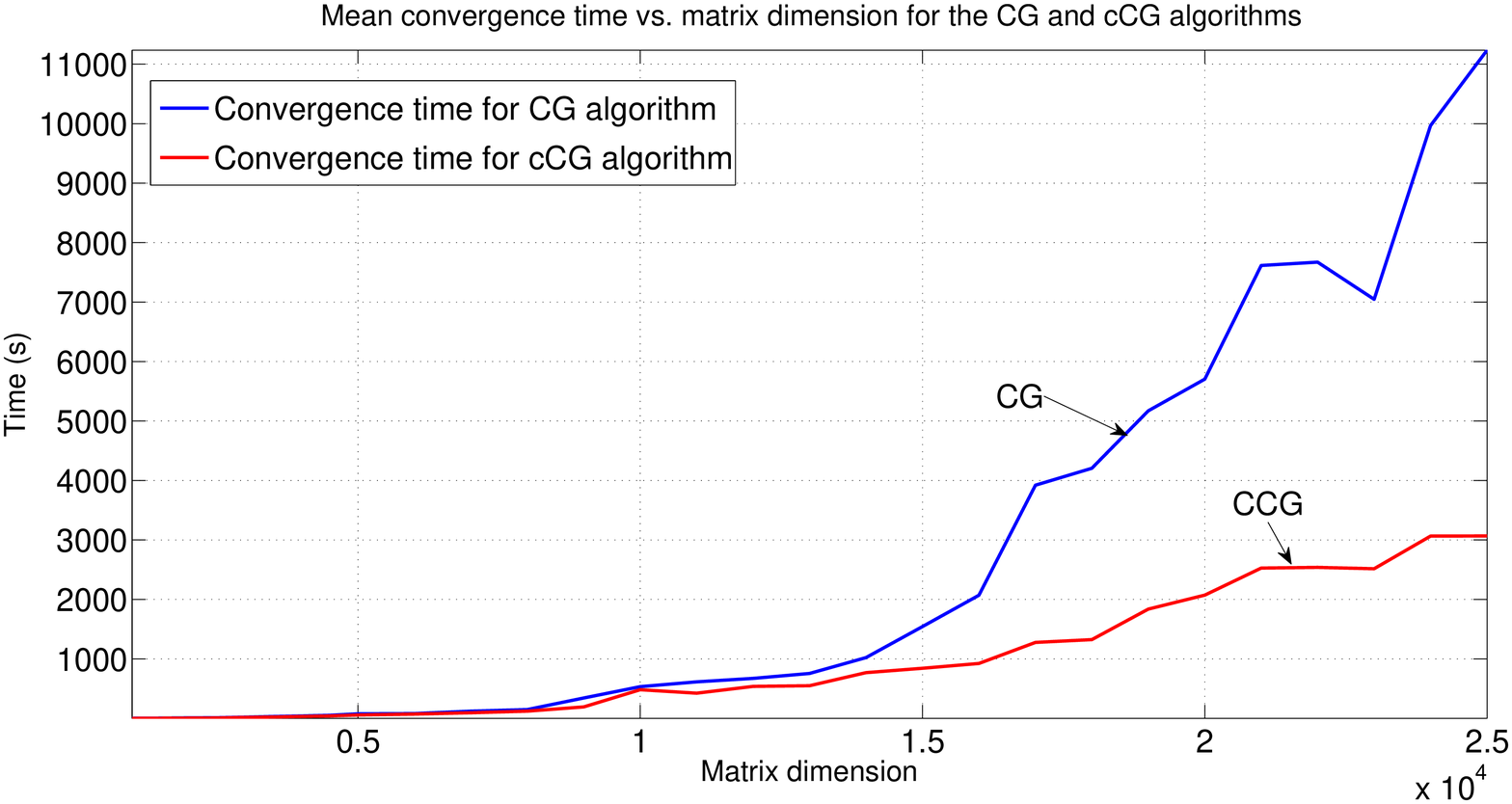}
\caption{Mean time to convergence for random test matrices of dimensions varying from $1000$ to $25000$, for 3 agent \cCG\ and standard \CG\ algorithms.}
\label{tab_results_g}
\end{figure}

The \textit{iteration speedup} of \cCG\ in comparison with \CG\ is defined as the number of iterations that \CG\ took to converge divided by the number of iterations \cCG\ took to converge and the experimental results are shown in Figure \ref{tab_gains_f}, which also shows the classical speed-up, which is the ratio of the time to convergence, i.e., the time taken to run the main loop until convergence, for \CG\ versus \cCG.

\begin{figure}[h]
\centering
\includegraphics[width=12cm]{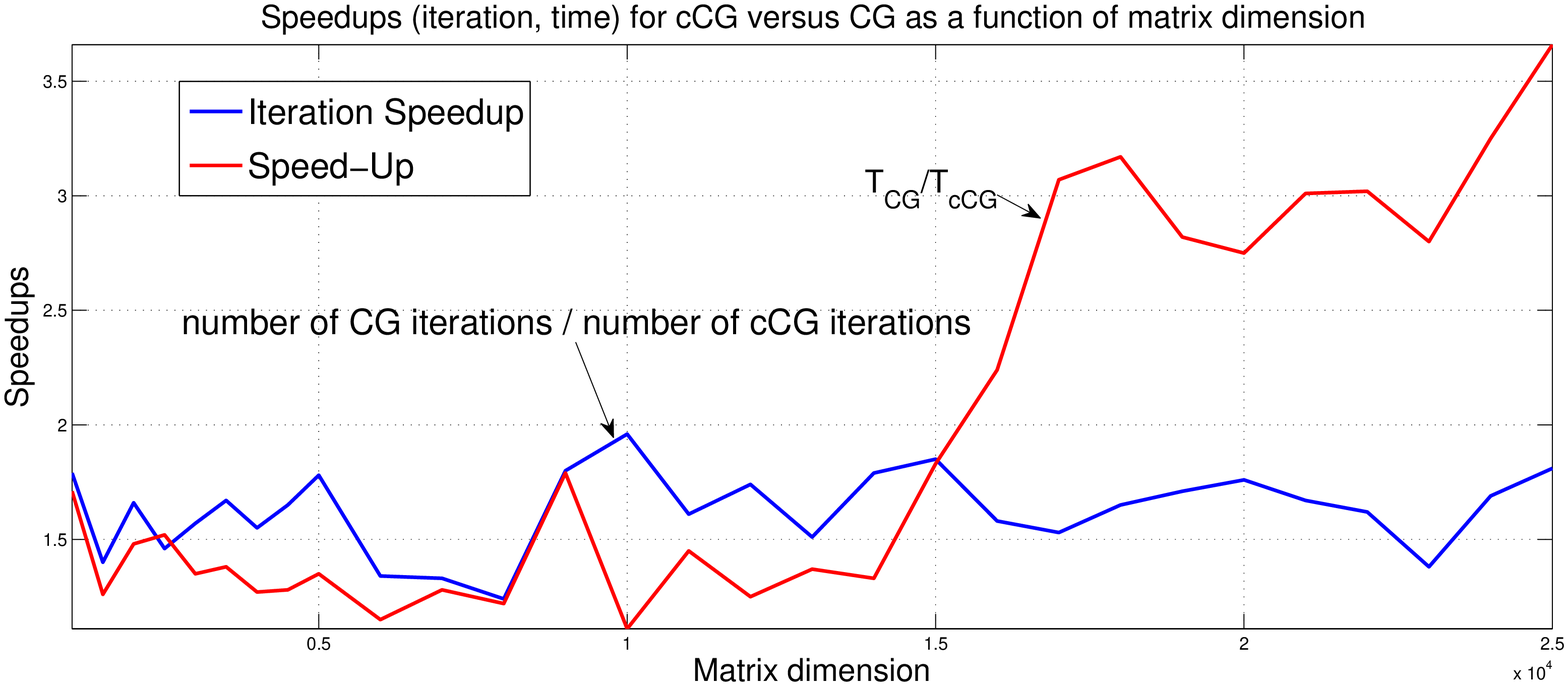}
\caption{Average speedups of Cooperative 3 agent \cCG\ over classic \CG\ for random test matrices of dimensions varying from $1000$ to $25000$.}
\label{tab_gains_f}
\end{figure}

The speedups seem to be roughly equal up to a certain size of matrix ($n=16000$); however, above this dimension, there is an increasing trend for both speedups.

The numerical results obtained show that \cCG, using 3 agents, leads to an improvement in comparison with the usual \CG\ algorithm. The average iteration speedup and the classical speedup of \cCG\ are respectively, $1.62$ and $1.94$, indicating that \cCG\ converges almost twice as fast as \CG\ for dense matrices with reasonably well-separated eigenvalues.

\subsection{Verifying the complexity estimates}
Figure 3 shows the mean time spent per iteration in seconds (points plotted as squares), versus matrix dimension, as well as the parabola fitted to this data, using least squares. Using the result from the last row of table 1 and multiplying it by the mean time per scalar multiplication, we obtain the parabola (dash-dotted line in Figure 3) expected in theory. In order to estimate the time per scalar multiplication, we divided the experimentally obtained mean total time spent on each iteration and divided it by the number of scalar multiplications performed in each iteration. This was done for each matrix dimension. Since the same multicore processor is being used for all experiments, each of these divisions should generate the same value of time taken to carry out each scalar multiplication, regardless of matrix dimension. It was observed that these divisions produced a data set which has a mean value of $8.10$ nanoseconds per scalar multiplication, with a standard deviation of $1.01$ nanoseconds, showing that the estimate is reasonable.

\begin{figure}[h]
\centering
\includegraphics[width=12cm]{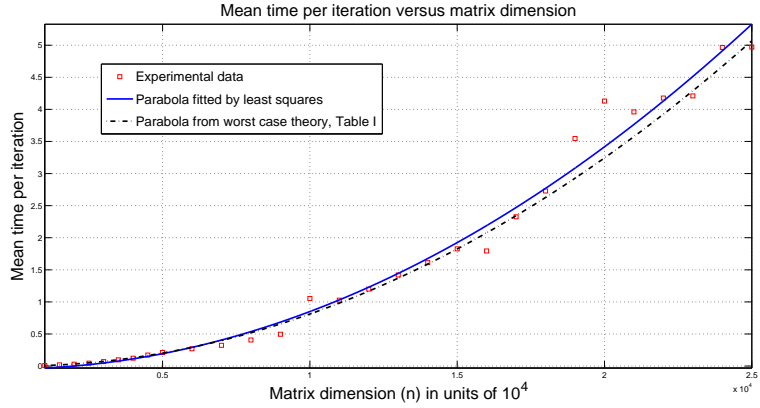}
\caption{Mean time per iteration versus problem dimension}
\label{parabolas}
\end{figure}
Now, from equation (\ref{worst_case}), substituting $p=3$, neglecting small order terms, and multiplying it by the estimated mean time per scalar multiplication ($8.10$ nanoseconds), the number of matrix multiplications per iteration, $N(p), p =3 $, is a cubic polynomial in $n$. Thus, the logarithm of the dimension ($n$) of the problem versus the logarithm of time needed to convergence is expected to be a straight line of slope $3$. Figure 4 shows this straight line, fitted to the data (squares) by least squares.

\begin{figure}[h]
\centering
\includegraphics[width=12cm]{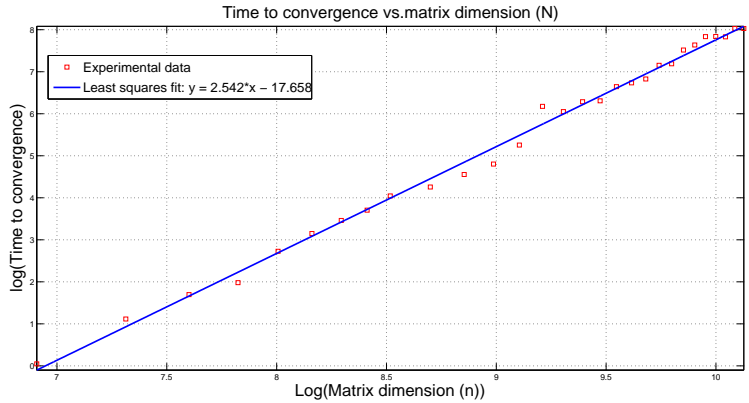}
\caption{Log-log plot of mean time to convergence versus problem dimension}
\label{time}
\end{figure}

Its slope ($2.542$) is fairly close to $3$, and data seems to follow a linear trend. The deviation of the slope from the ideal value has several probable causes, the first one being that the exact exponent of $3$ is a result of a worst case analysis of \CG\ in exact arithmetic. It is known that \CG\ usually converges, to a reasonable tolerance, in much less than $n$ iterations, where $n$ is the matrix dimension \cite{MS06}.

Similarly, the logarithm of the number of iterations needed to convergence versus the logarithm of the dimension of the problem should also follow a linear trend. Since the number of iterations is expected to be $n/3$, the slope of this line should be $1$. This log-log plot is shown in figure \ref{iterations}, in which the straight line was fitted by least squares to the original data (red squares). The slope ($0.501$) of the fitted line is smaller than $1$, but is seen to fit the data well (small residuals). The fact that both slopes are smaller than their expected values indicates that the \cCG\ algorithm is converging faster than the worst case estimate. Another reason is that a fairly coarse tolerance of $10^{-3}$ is used, and experiments reported show that decreasing the tolerance favors the \cCG\ algorithm even more. Specifically, for a randomly generated matrix  of dimension $8000$ and condition number $10^6$, Table 2 shows the mean number of iterations and time to convergence, calculated for $10$ different initial conditions, for the \CG and \cCG\ algorithms, as the tolerance is varied from $10^{-3}$ to $10^{-9}$

\begin{figure}[h]
\centering
\includegraphics[width=12cm]{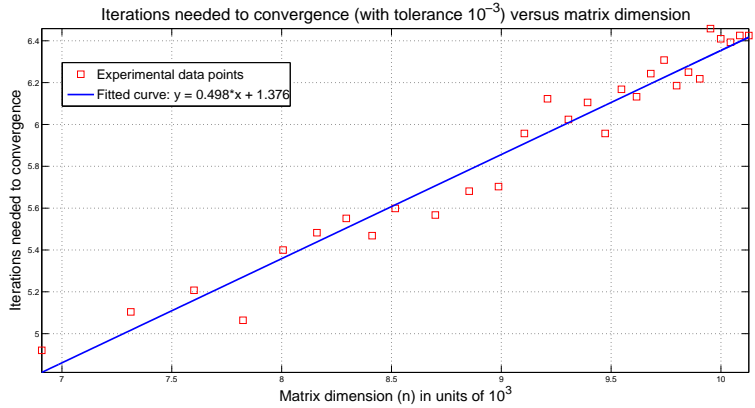}
\caption{Iterations needed to convergence versus problem dimension}
\label{iterations}
\end{figure}

  \begin{table}[h!]
    \centering
    \caption{Mean number of iterations and mean time to convergence for the \CG\ and the \cCG\ algorithms, as a function of tolerance used in the stopping criterion.}
    \begin{tabular}{| c | c | c | c | c |}
        \hline
        \textbf{Tolerance} & \multicolumn{2}{|c|}{\textbf{CG}} & \multicolumn{2}{|c|}{\textbf{cCG}} \\
        \cline{2-5} & Time (s) & Iterations & Time (s) & Iterations \\
        \hline
        $10^{-3}$ & $148.20$ & $372.20$ & $121.80$ & $299.70$ \\
        \hline
        $10^{-4}$ & $165.20$ & $414.00$ & $125.70$ & $319.40$ \\
        \hline
        $10^{-5}$ & $177.40$ & $444.90$ & $132.00$ & $335.00$ \\
        \hline
        $10^{-6}$ & $192.70$ & $476.80$ & $134.00$ & $352.70$ \\
        \hline
        $10^{-7}$ & $206.50$ & $510.70$ & $135.20$ & $336.20$ \\
        \hline
        $10^{-8}$ & $235.10$ & $538.40$ & $137.10$ & $373.50$ \\
        \hline
        $10^{-9}$ & $275.00$ & $559.70$ & $141.90$ & $381.20$ \\
        \hline
    \end{tabular}
    \label{tolerancia}
  \end{table}



The data used to generate all the graphs in the figures above is shown in tables 3 and 4.


\begin{table}[h!]
    \centering
        \begin{tabular}{| c | c | c | c | c |}
        \hline
            \textbf{Matrix Dimension} & \multicolumn{2}{|c|}{\textbf{Number of iterations}} & \multicolumn{2}{|c|}{\textbf{Time (s)}}\\
            \cline{2-5}& \textbf{CG} & \textbf{CCG} &\textbf{CG} &\textbf{CCG}\\
        \hline
            $1000$ & $245.50$ & $137.05$ & $1.80$ & $1.05$\\
        \hline
            $1500$ & $230.15$ & $146.65$ & $3.85$ & $3.05$\\
        \hline
            $2000$ & $303.35$ & $182.60$ & $8.05$ & $5.45$\\
        \hline
            $2500$ & $231.60$ & $158.20$ & $11.05$ & $7.25$\\
        \hline
            $3000$ & $347.00$ & $221.35$ & $20.65$ & $15.30$\\
        \hline
            $3500$ & $402.15$ & $240.50$ & $32.15$ & $23.35$\\
        \hline
            $4000$ & $399.85$ & $257.45$ & $40.30$ & $31.85$\\
        \hline
            $4500$ & $391.85$ & $237.10$ & $51.95$ & $40.70$\\
        \hline
            $5000$ & $481.60$ & $270.05$ & $77.00$ & $57.05$\\
        \hline
            $6000$ & $351.90$ & $261.70$ & $81.10$ & $70.60$\\
        \hline
            $7000$ & $390.90$ & $293.30$ & $121.70$ & $95.00$\\
        \hline
            $8000$ & $372.20$ & $299.70$ & $148.20$ & $121.80$\\
        \hline
            $9000$ & $659.90$ & $386.50$ & $343.50$ & $191.50$\\
        \hline
            $10000$ & $894.70$ & $456.20$ & $532.60$ & $480.80$\\
        \hline
            $11000$ & $667.00$ & $413.20$ & $614.40$ & $413.20$\\
        \hline
            $12000$ & $780.00$ & $448.20$ & $673.00$ & $537.40$\\
        \hline
            $13000$ & $582.80$ & $386.50$ & $753.90$ & $548.60$\\
        \hline
            $14000$ & $853.20$ & $477.30$ & $1022.70$ & $769.00$\\
        \hline
            $15000$ & $852.40$ & $460.60$ & $1543.00$ & $841.70$\\
        \hline
            $16000$ & $813.60$ & $514.40$ & $2070.00$ & $922.70$\\
        \hline
            $17000$ & $842.20$ & $548.90$ & $3921.60$ & $1277.20$\\
        \hline
            $18000$ & $802.50$ & $485.70$ & $4204.70$ & $1325.20$\\
        \hline
            $19000$ & $884.50$ & $518.00$ & $5171.50$ & $1836.70$\\
        \hline
            $20000$ & $882.10$ & $501.90$ & $5703.70$ & $2072.60$\\
        \hline
            $21000$ & $1064.30$ & $638.00$ & $7614.70$ & $2526.40$\\
        \hline
            $22000$ & $7671.80$ & $2537.60$ & $985.10$ & $607.60$\\
        \hline
            $23000$ & $7045.60$ & $2516.40$ & $826.30$ & $597.70$\\
        \hline
            $24000$ & $9969.20$ & $3065.20$ & $1040.90$ & $617.40$\\
        \hline
            $25000$ & $1114.70$ & $617.20$ & $11237.40$ & $3067.50$\\
        \hline
        \end{tabular}
    \caption{Average results for multiple test matrices of dimensions varying from $1000$ to $25000$, for Cooperative 3 agent \cCG and for classic \CG.}\label{tab_results}
\end{table}


\begin{table}[h!]
    \centering
        \begin{tabular}{| c | c | c |}
        \hline
            \textbf{Matrix Dimension} & \textbf{Iteration Ratio} & \textbf{Speed-Up}\\
        \hline
            $1000$ & $1.79$ & $1.71$\\
        \hline
            $1500$ & $1.40$ & $1.26$\\
        \hline
            $2000$ & $1.66$ & $1.48$\\
        \hline
            $2500$ & $1.46$ & $1.52$\\
        \hline
            $3000$ & $1.57$ & $1.35$\\
        \hline
            $3500$ & $1.67$ & $1.38$\\
        \hline
            $4000$ & $1.55$ & $1.27$\\
        \hline
            $4500$ & $1.65$ & $1.28$\\
        \hline
            $5000$ & $1.78$ & $1.35$\\
        \hline
            $6000$ & $1.34$ & $1.15$\\
        \hline
            $7000$ & $1.33$ & $1.28$\\
        \hline
            $8000$ & $1.24$ & $1.22$\\
        \hline
            $9000$ & $1.80$ & $1.79$\\
        \hline
            $10000$ & $1.96$ & $1.11$\\
        \hline
            $11000$ & $1.61$ & $1.45$\\
        \hline
            $12000$ & $1.74$ & $1.25$\\
        \hline
            $13000$ & $1.51$ & $1.37$\\
        \hline
            $14000$ & $1.79$ & $1.33$\\
        \hline
            $15000$ & $1.85$ & $1.83$\\
        \hline
            $16000$ & $1.58$ & $2.24$\\
        \hline
            $17000$ & $1.53$ & $3.07$\\
        \hline
            $18000$ & $1.65$ & $3.17$\\
        \hline
            $19000$ & $1.71$ & $2.82$\\
        \hline
            $20000$ & $1.76$ & $2.75$\\
        \hline
            $21000$ & $1.67$ & $3.01$\\
        \hline
            $22000$ & $1.62$ & $3.02$\\
        \hline
            $23000$ & $1.38$ & $2.80$\\
        \hline
            $24000$ & $1.69$ & $3.25$\\
        \hline
            $25000$ & $1.81$ & $3.66$\\
        \hline
        \end{tabular}
    \caption{Average gains of Cooperative 3 agent \cCG\ over classic \CG\ for matrices of dimensions varying from $1000$ to $25000$.}\label{tab_gains}
\end{table}

\section{Concluding Remarks}
\label{se6}
This paper proposed a new cooperative conjugate gradient (\cCG) method for linear systems with symmetric positive definite coefficient matrices. This \cCG\ method permits efficient implementation on a multicore computer and experimental results bear out the main theoretical properties, namely, that speedups close to the theoretical value of $p$, when a $p$-core computer is used, are possible, when the matrix dimension is suitably large. The experimental results of the current study were limited to dense randomly generated matrices and only $3$ cores of a $4$ core computer with a relatively small on-chip shared memory were used. Future work will include the study of the method on matrices that come from real applications and are typically sparse and sometimes ill-conditioned (which will necessitate the use of preconditioners) on larger multi-core machines. The use of larger machines should also permit exploration of the notable theoretical result (corollary \ref{cor2}) that, in the asymptotic limit, as $n$ becomes large, implying that $p$ also increases according to (\ref{appra}), solution of $Ax=b$ is possible by the method proposed here with a cost of $O(n^{2+\frac{1}{3}})$ multiplications. 


\end{document}